\documentclass[a4, 11pt] {amsart} 
\usepackage{amssymb,times, amscd,amsmath,amsthm, xypic}
\usepackage{geometry}
\usepackage{amssymb}
\usepackage[normalem]{ulem}
\usepackage{tikz}
\usepackage{tikz-cd}
\usepackage{hyperref}
\usepackage{thmtools, thm-restate}
\usepackage{ulem}
\usepackage{cleveref}
\numberwithin{equation}{section}
\newtheorem{pro}[equation]{Proposition}

\newtheorem{cor}[equation]{Corollary}

\newtheorem{lem}[equation]{Lemma}
\theoremstyle{definition}
\newtheorem{ex}[equation]{Example}
\newtheorem{defn}[equation]{Definition}

\newtheorem{rem}[equation]{Remark}

\newcommand{\cat}[1]{\mathsf{#1}}
\newcommand{\fun}[1]{\mathsf{#1}}
\newcommand{\id}{\mathrm{id}}
\newcommand{\poset}{\cat{Poset}}
\newcommand{\proset}{\cat{Proset}}
\newcommand{\tsp}{\cat{Top}}
\newcommand{\alex}{\cat{Alex}}

\newcommand{\asp}{\fun{T}}
\newcommand{\spp}{\fun{P}}

\def \bC{\mathbb C}

\def\bN{\mathbb N}
\def\bR{\mathbb R}

\def\bC{\mathbb  C}

\def\op{\operatorname}

\begin{document}

\title{ On stratifications and poset-stratified spaces} 

\thanks {
\noindent
\emph{keywords} : stratification, poset-stratified space, decomposition space, poset, Alexandrov topology }

\author{Lukas Waas, Jon Woolf and Shoji Yokura }

\address{Ruprecht--Karls--Universit\"at Heidelberg, Mathematisches Institut, 
Im Neuenheimer Feld 205, 69120 Heidelberg, Deutschland}
\email{lwaas@mathi.uni-heidelberg.de}

\address{Department of Mathematical Sciences, University of Liverpool, L69 7ZL United Kingdom} 
\email{Jonathan.Woolf@liverpool.ac.uk}

\address{Graduate School of Science and Engineering, Kagoshima University, 1-21-35 Korimoto, Kagoshima, 890-0065, Japan}
\email{yokura@sci.kagoshima-u.ac.jp}

\begin{abstract}
A stratified space is a topological space equipped with a \emph{stratification}, which is a decomposition or partition of the topological space satisfying certain extra conditions. More recently, the notion of poset-stratified space, i.e., 
topological space endowed with a continuous map to a poset with its Alexandrov topology, has been popularized.
Both notions of stratified spaces are ubiquitous in mathematics, ranging from investigations of singular structures in algebraic geometry to extensions of the homotopy hypothesis in higher category theory. In this article we study the precise mathematical relation between these different approaches to stratified spaces.
\end{abstract}

\maketitle

\section{Introduction}

Stratified spaces are ubiquitous in topology and geometry. For example, they arise naturally in the presence of group actions, singularities, and when the space considered is a configuration or moduli space. Classically, a stratification is a decomposition of the space satisfying certain conditions. More recently, stratified spaces have been defined as spaces over posets, again usually satisfying certain conditions. The purpose of this paper is to explain the relationship between these two perspectives.

We begin by recalling the classical definition, as given for example in \cite[8.2.4 Stratifications and Thom's $a_f$-Regularity]{Oka}. This defines the basic structure which underlies any flavour of stratified space. Depending on context, extra topological or geometric conditions may be imposed on the strata, and on how they fit together. Common examples of these extra structures include
\begin{itemize}
\item topological stratifications: each stratum is a \emph{topological manifold},
\item smooth stratifications: each stratum  is a \emph{smooth manifold},
\item Whitney stratifications \cite{Wh}: each pair of smooth strata satisfies \emph{Whitney's conditions A and B} concerning \emph{limits of tangent spaces and secant lines},
\item Thom--Mather stratifications \cite{Thom, Mather, Mather2}: each stratum  is smooth and is equipped with `\emph{control data}' describing a neighborhood of it.
\end{itemize}
\begin{defn}
A \emph{decomposition} of a topological space $X$ is a set $\{X_i\mid i\in I\}$ of non-empty, disjoint subspaces $X_i$, referred to as the \emph{strata} of the decomposition, such that $\bigcup_{i\in I}X_i = X$.
\end{defn}
\begin{defn}\label{stra-1}
A {\em stratification} of a topological space $X$ is a decomposition $\{X_i\mid i\in I\}$  satisfying:
\begin{description}
\item[Local finiteness] each $x \in X$ has an open neighborhood $U_x$ intersecting only finitely many strata, i.e., $\{i \in I \, \, | \, \, U_x \cap X_i \not = \emptyset \}$ is finite,
\item[Local closure] each $X_i$ is \emph{locally closed}, i.e. an intersection of an open and a closed set,
\item[Frontier condition]  if $X_i \cap \overline{X_j} \not = \emptyset$ then $X_i \subset \overline{X_j}$.
\end{description}
A topological space endowed with a stratification is called a \emph{stratified space}.
\end{defn}
The name of the last condition above refers to the fact that $X_i$ is contained in the \emph{frontier}  $\overline{X_j} \setminus X_j$ of $X_j$ when $i\neq j$. It is equivalent to the condition that \emph{the closure of each stratum is a union of strata}, and this alternative formulation is sometimes used as the definition, see, for example, \cite[Appendix C. Stratified Spaces and Singularities, Remark C.2., p. 434]{PS}. When the frontier condition holds, the set $I$ of strata is partially-ordered by the relation
\[
i\leq j \iff X_i \subseteq \overline{X_j}.
\]
This feature is sometimes incorporated into the definition of stratification, for example \cite{encyclopedia} defines an \emph{$I$-decomposition} of a space $X$ to be a stratification, as in Definition \ref{stra-1}, for which the set of strata $I$ is partially-ordered by $i \leq j \iff X_i \subset \overline{X_j}$. Thus a stratification with set of strata $I$ is an $I$-decomposition, and {\it vice versa}. Another minor variant of the definition is found in \cite{Stack}:
\begin{itemize}
\item a \emph{partition} is a decomposition with \emph{locally closed} strata;
\item a \emph{stratification} is a partition $\{X_i \mid i\in I\}$ of $X$ together with a \emph{partial order} on $I$ such that $\overline{X_j} \subset  \bigcup_{i \leq j} X_i$ for each $j \in I$;
\item a  \emph{good stratification} is a partition \emph{satisfying the frontier condition};
\end{itemize}
A good stratification is a stratification in the above sense --- the frontier condition implies that the set $I$ of strata is partially-ordered with $\overline{X_j}=\bigcup_{i\leq j} X_i$ --- and is a stratification in the sense of Definition \ref{stra-1} if, in addition, the associated decomposition  is locally finite.

\begin{rem}
 In \cite{Mather2}, see also \cite{wiki-stra}, John Mather uses the name ``\emph{prestratification}" for the above decomposition as in Definition \ref{stra-1} and the name ``stratification" is used for an existence of prestratifications locally at each point. 
\end{rem}

More recently, the partial ordering of the set of strata has been elevated to a central role leading to the notion of poset-stratified space. This perspective appears in \cite{Woolf} and was made popular by Jacob Lurie \cite{Lurie}, but the idea is already implicit in earlier work such as \cite{Hughes}.
\begin{defn}
\label{stra-2}
A \emph{poset-stratified space}  is a continuous map $\pi \colon X \to I_{\preccurlyeq}$ from a topological space $X$ 
to the Alexandrov space $I_{\preccurlyeq}$ of a poset $(I,\preccurlyeq)$.
\end{defn}

This definition has several advantages, not least that it has good categorical properties, and has been used extensively in studying the homotopy theory of stratified spaces, see for example \cite{BGH, Dou, DW, Ha, Hughes, Lurie, Nan, Wa1, Wa2, Wa3, Woolf, Woolf2}. 
The question arises as to what the precise relationship between stratified (in the sense of Definition \ref{stra-1}) and poset-stratified spaces is. In order to answer this we make the following definitions. 
\begin{defn}
Let $\{X_i \mid i\in I\}$ be a decomposition of $X$. The \emph{decomposition map} $\pi \colon X \to I$ is defined by $\pi(x)=i$ for all $x\in X_i$. The \emph{decomposition space}  $I_\pi$ is the set $I$ of strata equipped with the quotient topology, i.e., the finest topology for which $\pi$ is continuous.
\end{defn}
A poset-stratified space $\pi \colon X \to I_{\preccurlyeq}$ gives rise to the natural decomposition $\{X_i \, | \, i \in \op{Im} \pi \}$ of $X$ with $X_i :=\pi^{-1}(i)$. Suppose for a second that $\pi$ is surjective\footnote{This is generally not a major assumption, as one can always replace the poset $(I,\preccurlyeq)$ with $\op{Im} \pi$, equipped with the inherited partial order. Then the induced map $X \to \op{Im} \pi$ is a surjective poset-stratified space with the same associated decomposition.}, and we can thus identify the underlying set $I$ of $I_{\preccurlyeq}$ with the indexing set $\op{Im} \pi$. On the level of sets, the associated decomposition map $\pi: X \to I_{\pi}$ agrees with $\pi$. However, on the level of the spaces, the topology on $I_{\pi}$ is, a priori, finer  than the Alexandrov topology on $I_{\preccurlyeq}$ associated to $\preccurlyeq$. The motivation of this present work (and \cite{Yo}) is to study the relation between ``stratification" and ``poset-stratified space" by \emph{analyzing properties of the continuous decomposition map $\pi: X \to I_{\pi}$}.
In answer to the question of the relationship between stratified and poset-stratified spaces our analysis of decompositions has the following three consequences. First, as is `well-known to experts' that, any stratified space is poset-stratified:
\begin{restatable}{thm}{thmA}
\label{thm A}
Suppose that $\{X_i \mid i\in I\}$ is a stratification of $X$. Then 
\begin{enumerate}
\item the set $I$ of strata is partially-ordered by $i\leq j \iff X_i \subseteq \overline{X_j}$ and 
\item the decomposition map $\pi \colon X \to I_{\pi} = I_\leq$ is a poset-stratified space with respect to this partial order.
\end{enumerate}
\end{restatable}

Second, not every poset-stratified space is stratified: the strata of a poset-stratified space are locally closed, but the associated decomposition need not be locally finite, nor satisfy the frontier condition. 
The next result gives sufficient conditions, in terms of the Alexandrov space of the poset and the map to it, for a poset-stratified space to be stratified. We derive this as a consequence of our analysis of decompositions; it can also be proved directly,  see Remark~\ref{rem:thm B 2} and Dai Tamaki and Hiro Lee Tanaka \cite[Propositions 3.2 and 3.4]{TT}.
\begin{restatable}{thm}{thmB}
\label{thm B}
Suppose $\pi \colon X \to I_\preccurlyeq$ is a poset-stratified space such that the Alexandrov space $I_\preccurlyeq$ is locally finite and $\pi$ is an open map. Then the decomposition $\{\pi^{-1}(i) \mid i \in \op{Im}\pi
\}$ is a stratification of $X$. 
\end{restatable}
When the conditions of Theorem \ref{thm B}  are satisfied, the decomposition $\{\pi^{-1}(i) \mid i\in \op{Im}\pi
\}$ is a stratification. Assuming $\op{Im} \pi = I$, it follows that $\pi \colon X \to I_\leq$ is a poset-stratified space with respect to the partial order 
\begin{equation}
\label{decomp order}
i\leq j \iff \pi^{-1}(i) \subseteq \overline{\pi^{-1}(j)}
\end{equation}
 by Theorem \ref{thm A}. In fact, in this case, the partial order $\leq$ is initial amongst those partial orders on the set $I$ of strata for which the decomposition map is a 
poset-stratified space, i.e.\ the set-theoretic identity $(I,\leq) \to (I,\preccurlyeq)$ is monotone\footnote{Equivalently, the identity map $I_{\leq} \to I_{\preccurlyeq}$ is continuous, since for any open set $U$ of $I_{\preccurlyeq}$ the inverse image $\pi^{-1}(U)$ is an open set by the continuity of $\pi \colon X \to I_\preccurlyeq$, thus $U$ is an open set in $I_{\pi}=I_{\leq}$ because $I_{\pi}$ has the quotient topology.}, but need not be an order isomorphism. 

The third consequence is that the sufficient conditions in Theorem \ref{thm B}  are necessary when $X$ is poset-stratified for the {\em decomposition preorder} in which 
\[
i \leq j \iff \pi^{-1}(i) \ \text{is in the minimal closed union of strata containing}\ \pi^{-1}(j).
\]
 Moreover, when this holds the decomposition preorder has the simpler description (\ref{decomp order}).
\begin{restatable}{thm}{thmC}
\label{thm C} Let $\{ X_i \mid i\in I\}$ be a decomposition of $X$.
Suppose $\pi \colon X \to I_\leq$ is poset-stratified with respect to the decomposition preorder  on $I$. Then $\{ X_i \mid i\in I\}$ is a stratification if, and only if,  the Alexandrov space $I_\leq$ is locally finite and the map $\pi$ is open.
\end{restatable}

In summary the contents of the paper are as follows. In \S \ref{sec:preorders and alexandrov spaces} we recall the notion of an Alexandrov space and review the well-known adjunction between preordered sets and topological spaces, which restricts to an equivalence between preorders and Alexandrov spaces. 

In \S\ref{sec:decompositions} we introduce the notion of an {\em Alexandrov decomposition}, that is one whose decomposition space is an Alexandrov space. This is a slightly weaker notion 
than that of a poset-stratified space, but which still turns out to have good properties, and forms a useful bridge between the worlds of stratified and poset-stratified spaces. In particular we show, Corollary \ref{cor:lf is alex}, that every locally finite decomposition is an Alexandrov decomposition. In fact we derive this as a consequence of a more general result, Proposition \ref{general-l.f.}, about decompositions of spaces with a final topology for a suitable family of maps. We end the section by showing that an Alexandrov decomposition is a poset-stratified space when each stratum is open in the minimal closed union of strata containing it (Proposition \ref{poset-strat}). This is a slightly stronger condition than asking that each stratum is locally closed. 

In \S\ref{sec:frontier condition} we discuss the frontier condition. This is the key property which ensures a  close relationship between the properties of a decomposition and of its decomposition map and space. For example, although the strata of a poset-stratified space are always locally closed, the converse is {\em not} true in the absence of the frontier condition. For example, the decomposition
\[
S^1 = \left\{ e^{\sqrt{-1} \theta} \mid 0< \theta \leq \pi\right\} \sqcup  \left\{ e^{\sqrt{-1} \theta} \mid \pi< \theta \leq 2\pi\right\}
\]
of the unit circle in $\bC$ has locally closed strata, but the decomposition space is the indiscrete space with two points. This is not the Alexandrov space of (any) partial-order on the set of strata, so the decomposition map does not exhibit $S^1$ as a poset-stratified space. The main result of this section is Proposition \ref{FC} which states that an Alexandrov decomposition satisfies the frontier condition if, and only if, its decomposition map is open. Both conditions fail in the above example.

Finally in \S\ref{sec:main} we assemble the results of \S\ref{sec:decompositions} and \ref{sec:frontier condition} to prove Theorems \ref{thm A},  \ref{thm B} and \ref{thm C}, and give a number of simple, but typical, examples to illustrate them.

In Appendix \ref{sec:semicontinuity} we briefly discuss Moore's notions of upper and lower semicontinuous decompositions. Lower semicontinuous decompositions are those for which the decomposition map is open, and so are closely related to stratifications and the contents of this paper. Upper semicontinuous decompositions appear in several famous results, including Moore's and Bing's theorems on decompositions of Euclidean spaces into continua, and Freedman's proof of the Poincar\'e Conjecture in 
dimension $4$. Loosely,  upper semicontinuous decompositions seem to appear in `wild topology' and lower semicontinuous ones in `stratification theory', which is often an attempt to tame the topology of singular spaces.

\subsection*{Acknowledgements}
We would like to express our thanks to the anonymous referee for his/her careful reading the paper and very useful and constructive suggestions and comments.
L.W. is supported by the Landesgraduiertenf\"orderung Baden-W\"urttemberg. 
S.Y. is supported by JSPS KAKENHI Grant Numbers JP19K03468 and JP23K03117. \\

\section{Preorders and Alexandrov Spaces}
\label{sec:preorders and alexandrov spaces}

A preorder on a set $P$ is a relation $\leq$ which is reflexive and transitive. A set $(P, \leq)$ equipped with a preorder $\leq$ is called a \emph{proset}. If a preorder is in addition anti-symmetric, then it is a partial order; a set with a partial order is called a \emph{poset}. A map $f\colon P \to Q$ between prosets is {\em monotone} if $p\leq p' \implies f(p) \leq f(p')$.

A preorder $\leq$ on $P$ defines an equivalence relation $p \sim p' \iff p\leq p'$ and $p'\leq p$. The set $P/\!\!\sim$ of equivalence classes inherits a natural partial order defined by $[p]\preccurlyeq [p'] \iff p\leq p'$. Let $\proset$ be the category of prosets and monotone maps, and  $\poset$ the full subcategory of posets. The assignment $P \mapsto P/\!\!\!\sim$ extends to a functor $\proset \to \poset$ which is left adjoint to the inclusion $\poset \hookrightarrow \proset$.

A preorder $\leq$ determines a natural topology on $P$ in which $U \subset P$ is open if, and only if, it is upward-closed, i.e., 
\[
p \in U\ \text{and}\  p \leq q \implies q \in U.
\]
Each $p\in P$ has a minimal open 
neighborhood   $U_p = \{q \in P \mid p \leq q \}$ in this topology, and the set of minimal open neighborhoods $\{U_p \, | \, p \in P\}$ is a base. This topology is referred to as the \emph{Alexandrov topology} on $P$ because it makes $P$ into an Alexandrov space in the sense of the following definition. We denote the space $P$ equipped with this topology by $P_\leq$. 

\begin{defn}[Alexandrov space]  
Let $X$ be a topological space. If the intersection of any family of open sets is open, equivalently the union of any family of closed sets is closed, then the topology is called an \emph{Alexandrov topology} and the space is called an \emph{Alexandrov space}. 
\end{defn}

The property of being an Alexandrov space is local.
\begin{pro} 
\label{pro:la is alex}
Any locally Alexandrov space, i.e., any space in which each point has an open
neighborhood which is an Alexandrov space in the subspace topology, is an Alexandrov space.
\end{pro}
\begin{proof} Let $U_i\subseteq X$ be open for $i\in I$. To show that $X$ is Alexandrov, we must show that $\bigcap_{i\in I}U_i$ is an open subset of $X$. Since $X$ is locally Alexandrov, there is an open covering $X =\bigcup_{j \in J}V_j$ such that each $V_j$ is an Alexandrov space in the subspace topology. Then
\[
\bigcap_{i\in I}U_i = X \cap \left(\bigcap_{i\in I}U_i\right) = \left(\bigcup_{j \in J}V_j\right) \cap  \left(\bigcap_{i\in I}U_i\right) = \bigcup_{j \in J} \left(\bigcap_{i \in I}(V_j \cap U_i)\right).
\]
Since each $V_j$ is an Alexandrov space, $\bigcap_{i \in I}(V_j \cap U_i)$ is an open subset of $V_j$, hence also an open subset of $X$. Thus the right hand side above is an open subset of $X$, and therefore so is $\bigcap_{i\in I}U_i$. 
\end{proof}
\begin{rem}
Any finite topological space, i.e., any topological space with finitely many points, is  Alexandrov. Therefore any locally finite space is locally Alexandrov, and hence Alexandrov by the above. This latter result is due to Ioan Mackenzie James, see \cite{James} and \cite[Corollary 3.5]{NO}. The converse is false: there are Alexandrov spaces which are not locally finite: for example, the Alexandrov space of the natural numbers $\bN$ with the usual order is \emph{not} locally finite, because the minimal open neighborhood $U_m$ of $m\in \bN$ is the infinite set $\{n \mid m\leq n \}$.
\end{rem}



The following lemma is well-known, but as it plays a crucial role in this paper we give a proof.
\begin{lem} A topological space $X$ is an Alexandrov space if, and only if, each point in $X$ has a minimal open neighborhood. 
\end{lem}
\begin{proof}
If $X$ is an Alexandrov space, then  the intersection $U_x$ of all the open neighborhoods of a point $x\in X$ is a minimal open neighborhood of $x$. Conversely, suppose that each point $x\in X$ has a minimal open neighborhood $U_x$. Let $U=\bigcap_{j \in J}U_j$ be an intersection of open sets. If $U=\emptyset$, then we are done. If $U \not = \emptyset$, then $U_x \subset U_j$ for all $j \in J$ and $x\in U$, so $x \in U_x \subset U$. Hence 
\[
U =\bigcap_{j\in J}U_j =  \bigcup_{x \in U} U_x
\]
is a union of open sets, and so is open. Thus  $X$ is an Alexandrov space.
\end{proof}

\begin{rem}
\begin{enumerate}
\item For a proset $(P, \leq)$ the downward-closed set  $D_p :=\{ q \, \, | \, \, q \leq p\}$ is the minimal \emph{closed} neighborhood of $p$. This is because its complement 
\[
P-D_p = \bigcup_{r\not \leq p} U_r
\]
 is the maximal open set not containing $p$.
\item The closed sets of an Alexandrov space also form an Alexandrov topology, because they are closed under arbitrary intersections and unions, equivalently because the reverse or opposite of a preorder is again a preorder. 
In this topology $U_p$ is the minimal closed, and $D_p$ the minimal open, 
neighborhood of $p$. This complementary topology is sometimes referred to as `the' topology on the proset, see for example \cite{Ar}, \cite{B}, \cite{May1} and \cite{Sp}. However, when stratification theory or poset-stratified spaces are considered, as in \cite{Curry} and \cite{Woolf}, it is more convenient to take $U_p$ to be open, see also \cite[Definition A.5.1]{Lurie} and \cite[Definition 2.1]{Tam}. 
\end{enumerate}
\end{rem}

\begin{pro}
\label{pro:proset alex adj}
There is an adjunction $\asp \colon \proset \longleftrightarrow \tsp \colon \spp$ between prosets and topological spaces. The left adjoint $\asp$ takes a proset to its Alexandrov space, and the right adjoint $\spp$ takes a topological space to its set of points equipped with the {\em specialization preorder} 
\[
x \leq y \iff x\in\overline{\{y\}}.
\]
The adjunction restricts to an equivalence $\asp \colon \proset \longleftrightarrow \alex \colon \spp$ between the category of prosets and the full subcategory of  Alexandrov spaces. 
\end{pro}
\begin{proof}[Sketch proof]
This is also well-known, so we provide only a sketch. The assignment $P \mapsto \asp(P)$ extends to a functor because any monotone map $f\colon P\to Q$ is continuous with respect to the Alexandrov topologies on $P$ and $Q$. Similarly, the assignment $X \mapsto \spp(X)$ extends to a functor because a continuous map $f \colon X\to Y$ is monotone as a map between the specialization preorders 
on $X$ and $Y$. 

The specialization preorder on the Alexandrov space $\asp(P)$ of a proset $P$ is the original preorder because $p\leq q \iff p \in D_q$ for both preorders. Thus the identity defines a natural isomorphism $P \cong (\spp\circ \asp)(P)$. In the other direction, the Alexandrov topology of the specialization preorder of a space $X$ is always finer than the original topology. To see this note that if $D\subseteq X$ is closed, $y\in D$ and $x\leq y$ then $x\in \overline{\{y\}} \subseteq D$, so that $D \subseteq (\asp\circ\spp)(X)$ is closed too. Thus the identity defines a natural map
\[
(\asp\circ\spp)(X) \to X.
\]
These natural maps are respectively the unit and counit of the adjunction. 

When $X$ is an Alexandrov space, then we can write any closed $D\subseteq X$ as the union $D=\bigcup_{x\in D} \overline{\{x\}}$ of closed subsets in $(\asp\circ\spp)(X)$. Therefore $D$ is closed in $(\asp\circ\spp)(X)$ and the identity $(\asp\circ\spp)(X) \cong X$ is a homeomorphism. Thus the adjunction restricts to an equivalence between prosets and Alexandrov spaces. 
\end{proof}
Finally, we give a characterization of  the Alexandrov spaces of posets which will be useful later.
\begin{lem}
\label{poset-lc}
\label{locally-closed} A proset $P$ is a poset if, and only if, each singleton $\{p\}$ is locally closed in the Alexandrov space $\asp(P)$.
\end{lem}
\begin{proof} 
Recall that each $p\in P$ has a minimal open 
neighborhood $U_p = \{ q\in P \mid p\leq q\}$ and a minimal closed 
neighborhood $D_p = \{q\in P \mid p\geq q\}$ in the Alexandrov space $\asp(P)$. Moreover,  $P$ is a poset precisely when $\{p\} = D_p \cap U_p$ for all $p\in P$. Therefore, if $P$ is a poset, each singleton is locally closed. Conversely, if $\{p\}$ is locally closed, then $\{p\}=D\cap U$ for some closed  $D$ and open $U$. The minimality of $D_p$ and $U_p$ then shows
\[
\{p\} = D\cap U \supset D_p \cap U_p \supset \{p\}.
\] 
Hence $\{p\}=D_p\cap U_p$ and each singleton is locally closed.
\end{proof}

For more on the Alexandrov topology or Alexandrov spaces see, e.g., \cite{A1}, \cite{A2}, \cite[\S 4.2.1 Alexandrov Topology]{Curry}, \cite[Appendix A  Pre-orders and spaces]{Woolf}.

\section{Alexandrov decompositions}
\label{sec:decompositions}

A \emph{decomposition} of a space $X$ is a partition into subspaces, i.e.\ a set $\{X_i \mid i\in I\}$ of disjoint, non-empty subspaces whose union is $X$. We refer to the $X_i$ as the {\em strata} of the decomposition, and (by a mild abuse of notation) to the indexing set $I$ as the {\em set of strata}. 

There are two equivalent descriptions which we will use. Firstly, a decomposition can be viewed as the equivalence relation 
\[
x \sim x' \iff x,x' \in X_i \ \text{for some}\ i\in I
\] 
with equivalence classes the strata, and set of equivalence classes  $X/\!\!\sim\  \cong  I$. Secondly, it can be viewed as the quotient map $\pi \colon X \to I$ defined by
\[
\pi(x) = i \iff x\in X_i.
\]
We refer to this as the {\em decomposition map}. The \emph{decomposition space} $I_\pi$ is the set $I$ of strata equipped with the quotient topology from $\pi$, i.e.\ with the finest topology (most open sets) such that $\pi$  is continuous. 

The set $I$ of strata has a second interesting topology. Let $I_\leq = \left(\asp\circ \spp\right)(I_\pi)$ be the Alexandrov space of the specialization preorder of $I_\pi$. We refer to this as the  \emph{decomposition preorder} on $I$. It is defined by $i\leq j \iff i \in \overline{\{j\}}$ where the latter denotes the closure in $I_\pi$. By definition 
\[
\overline{\{j\}} = \{ i \in I \mid i\leq j\} = D_j
\]
is the minimal (closed) subset containing 
$j$ whose preimage in $X$ is closed. Equivalently, 
$\pi^{-1}(D_j)$ is the minimal closed union of strata containing 
$X_j$. Therefore, in terms of the decomposition 
\[
i \leq j \iff X_i\ \text{is contained in the minimal closed union of strata containing}\ X_j.
\]
 This is strictly weaker than $X_i \cap \overline{X_j} \neq \emptyset$: for example the decomposition $\{0\} \cup (0,1]\cup (1,2]$ of the closed interval $[0,2]$ has decomposition preorder $0 \leq 1 \leq 2$ even though $\{0\} \cap \overline{(1,2]} =\emptyset$. Here $X_0 := \{0\}, X_1 := (0,1]$ and $X_2 :=(1, 2]$. 

The above two topologies on $I$ are related: The set-theoretic identity $I_\leq =\left(\asp\circ \spp\right)(I_\pi) \to I_\pi$ is the counit of the adjunction $\asp \dashv \spp$, with $\asp$ the fully faithful inclusion of prosets into spaces. Hence, this map is generally continuous, and a homeomorphism precisely when $I_\pi$ is an Alexandrov space.
\begin{defn}
A decomposition $\{ X_i \mid i\in I\}$ of $X$ is \emph{Alexandrov} if its decomposition space $I_\pi$ is an Alexandrov space.
\end{defn}
\begin{pro}
\label{pro-alex}
The following are equivalent:
\begin{enumerate}
\item the decomposition $\{ X_i \mid i\in I\}$ of $X$ is Alexandrov,
\item the identity $I_\leq \to I_\pi$ is a homeomorphism,
\item the map $\pi \colon X\to I_\leq$ is continuous.
\end{enumerate}
\end{pro}
\begin{proof}
(1) $\implies$ (2) because $I_\leq  = \left(\asp\circ \spp\right)(I_\pi) = I_\pi$ when $I_\pi$ is Alexandrov. It is immediate that (2) $\implies$ (1)  and (3). Finally, (3) $\implies$ (2) because the quotient topology is initial amongst topologies on $I$ for which $\pi$ is continuous, i.e.\ when $\pi \colon X \to I_\leq$ there is a commutative diagram
\[
\xymatrix
{
X \ar[d]_{\pi} \ar[dr]^{\pi} & \\
I_\pi \ar[r]_{\id} & I_\leq
}
\]
of continuous maps. Since the identity $I_\leq \to I_\pi$ is also continuous, it is a homeomorphism. 
\end{proof}

The pointwise decomposition $\{ \{x\} \mid x\in X\}$ of a space $X$ is Alexandrov if, and only if, the space $X$ is Alexandrov (because the decomposition space is just $X$ itself). For example, the pointwise decomposition of the real line is not Alexandrov. In contrast, every finite decomposition is Alexandrov, because every finite topological space is evidently an Alexandrov space. The following result generalizes this example to give a wide class of Alexandrov decompositions.

\begin{pro}
\label{general-l.f.} 
Let $\{ X_i \mid i\in I\}$ be a decomposition of $X$. Suppose $X$ has the final topology with respect to a family of maps $f_j : Y_j \to X$ for $j\in J$ such that $\{ i\in I \mid f_j^{-1}(X_i) \neq \emptyset\}$ is finite for each $j\in J$. Then the decomposition is Alexandrov. 
\end{pro}

Recall that the \emph{final topology} is the finest topology on $X$ such that $f_j : Y_j \to X$ is continuous for each $j\in J$. Explicitly, $U\subseteq X$ is open in the final topology if, and only if,  $f_j^{-1}(U)$ is open in $Y_j$ for all $j\in J$. For example, the final topology for a family consisting of a single map $f\colon X \to Y$ is the quotient topology. The final topology has the universal property that  $g\colon X \to Z$ is continuous if, and only if, the composite $g \circ f_j \colon Y_j \to Z$ is continuous for each $j\in J$. Before proving Proposition \ref{general-l.f.},  we give two corollaries.

\begin{cor}
\label{cor:lf is alex}
Any locally finite decomposition is Alexandrov.
\end{cor}
\begin{proof}
Choose an open covering $X=\bigcup_{j\in J}U_j$ by subsets $U_j$ intersecting only finitely many strata. The original topology of $X$ is the final topology of $X$ for the family of the inclusions $\imath_i: U_i \to X$ (this is true for any open covering, without requiring the above condition `subsets $U_j$ intersecting only finitely many strata'). So the decomposition is Alexandrov by Proposition \ref{general-l.f.}.
\end{proof}
\begin{cor}
\label{cor:compact finite is alex}
Any decomposition of a compactly generated space for which each compact subspace intersects only finitely many strata is Alexandrov.
\end{cor}
\begin{proof}
Recall that a space $X$ is compactly generated\footnote{As to this notion, there are several slightly different definitions, e.g., see \cite{nLab, wiki-cg}.}
if a subset $U$ is open in $X$ if and only if $U \cap K$ is open in $K$ for every compact $K \subset X$. In other words, $X$ is compactly generated if it has the final topology for the family of inclusions of compact subspaces. Thus, if each compact subspace intersects only finitely many strata,  the decomposition is Alexandrov by Proposition \ref{general-l.f.}.
\end{proof}

We now prove Proposition \ref{general-l.f.} as a consequence of the following two lemmas.
\begin{lem}
\label{final-lem1}
Let $f \colon X' \to X$ be a continuous map. Let $\{ X_i \mid i\in I\}$ be a decomposition of $X$, and $\{f^{-1}(X_i) \mid i\in I' \}$ the induced decomposition of $X'$ with indexing set
\[
I' = \{ i\in I \mid f^{-1}(X_i)\neq \emptyset\}.
\]
Then the inclusion $e\colon I'\hookrightarrow I$ is continuous in the Alexandrov topologies of the  decomposition preorders on $I'$ and $I$.
\end{lem}
\begin{proof}
By definition there is a commutative diagram
\[
\xymatrix
{
X' \ar[r]^{f} \ar[d]_{\pi'} & X \ar[d]^{\pi}\\
I' \ar@{^{(}-_>}[r]^e & I
}
\]
where $\pi'$ and $\pi$ are the respective decomposition maps. Thus the inclusion $e$ is continuous in the quotient topologies on $I'$ and $I$ respectively, that is, $e \colon I'_{\pi'} \hookrightarrow I_\pi$  is continuous. Applying the functor $\asp \circ \spp$  to this map gives the result. 
\end{proof}

\begin{lem}
\label{final-lem2}
Suppose $X$ has the final topology with respect to a family of maps $f_j : Y_j \to X$ for $j\in J$. Then a decomposition  $\{X_i \mid i\in I\}$ of $X$ is Alexandrov if, and only if, the induced decomposition $\{ f_j^{-1}(X_i) \mid i\in I_j\}$ of each $Y_j$ is Alexandrov, where $I_j = \{ i\in I \mid f_j^{-1}(X_i) \neq \emptyset\}$.
\end{lem}
\begin{proof}
For each $j\in J$  let   $\pi_j \colon Y_j \to I_j$ be the decomposition map of the induced decomposition, and $e_j \colon I_j \to I$ the inclusion. By definition the diagram
\[
\xymatrix
{
Y_j \ar[r]^{f_j} \ar[d]_{\pi_j} & X \ar[d]^{\pi}\\
I_j \ar@{^{(}-_>}[r]^{e_j} & I
}
\]
commutes. By Lemma \ref{final-lem1} the inclusion $e_j$ is continuous in the Alexandrov topologies from the respective decomposition preorders on $I_j$ and $I$. Using Proposition \ref{pro-alex} and  the facts that $X$ has the final topology with respect to the family of maps $f_j : Y_j \to X$ for $j\in J$ , and that a map to a subspace is continuous if, and only if, the composite with the inclusion of the subspace is continuous, we have
\begin{align*}
\pi \colon X \to I_\leq \ \text{is continuous}
&\iff \pi\circ f_j \colon Y_j \to I_\leq \ \text{is continuous for all}\ j\in J \\
&\iff e_j \circ \pi_j \colon Y_j \to I_\leq \ \text{is continuous for all}\ j\in J\\
& \iff \pi_j \colon Y_j \to (I_j)_\leq \ \text{is continuous for all}\ j\in J.
\end{align*}
Hence, by Proposition \ref{pro-alex}, the decomposition of $X$ is Alexandrov if, and only if, the induced decompositions of each of the $Y_j$ are Alexandrov.
\end{proof}

\begin{proof}[Proof of Proposition \ref{general-l.f.}]
This follows immediately from Lemma \ref{final-lem2} and the fact that any finite decomposition is Alexandrov. 
\end{proof}

Alexandrov decompositions are closely related to poset-stratified spaces. Indeed, any Alexandrov decomposition $\{X_i \mid i\in I\}$ has a natural coarsening which is poset-stratified. Let $(I,\leq)$ be the proset with the decomposition preorder and let $I/\sim$ be the set of equivalence classes of $i \sim j \iff i\leq j$ and $j\leq i$,  equipped with the induced partial order $[i] \preccurlyeq [j] \iff i \leq j$ (cf. \cite[\S 5]{Yo}). The quotient map  $I \to I/\!\sim \colon i \mapsto [i]$ is monotone, thus the assoicated map $I_{\leq} \to (I/\sim)_\preccurlyeq$ is continuous, so the composite
\[
\xymatrix
{X \ar[r]^{\pi} & I_\leq \ar[r]{} & (I/\sim)_\preccurlyeq
}
\]
is continuous. This exhibits $X$ as a poset-stratified space: its strata $X_{[i]} = \bigcup_{j\in [i]} X_j$ are unions of strata of the original decomposition. These coarsened strata have a natural topological description.
Recall that $D_i = \{ j\in I \mid j\leq i\}$ is the closure of $\{i\}$ in $I_\pi$  and that $\pi^{-1}(D_i)$ is the minimal closed union of strata containing $X_i$. Then 
\[
X_{[i]} = X_{[j]} \iff \pi^{-1}(D_i) = \pi^{-1}(D_j)
\]
i.e., the minimal closed unions of strata containing $X_i$ and $X_j$ agree. 

We are interested in when the original stratification is poset-stratified, so we now consider conditions under which the decomposition preorder is a partial order. 
\begin{defn}\label{stra-2-1}
A decomposition $X = \bigsqcup_{i\in I} X_i$ is a \emph{poset-stratification} with respect to a partial order $\preccurlyeq$ on the set $I$ of strata if the decomposition map $\pi \colon X \to I_\preccurlyeq$ is continuous with respect to the Alexandrov topology defined by $\preccurlyeq$. 
\end{defn}
It is important to note that, just as a poset-stratified space need not be stratified in the classical sense, a poset-stratification need not be a stratification in the sense of Definition~\ref{stra-1}.
\begin{pro}
\label{poset-strat}
Suppose $\{X_i \mid i\in I\}$ is an Alexandrov decomposition of $X$. Then the following are equivalent:
\begin{enumerate}
\item the decomposition is a poset-stratification with respect to a partial order $\preccurlyeq$ on $I$,
\item the decomposition preorder $\leq$ is a partial order and the decomposition is a poset-stratification with respect to $\leq$,
\item each stratum $X_i$ is open in  $\pi^{-1}(D_i)$, in particular each stratum is locally closed.
\end{enumerate}
When these equivalent conditions hold, 
the identity $(I,\leq) \to (I,\preccurlyeq)$ is monotone. 
\end{pro}
\begin{proof}
(1) $\iff$ (2): Suppose $\pi \colon X \to I_\preccurlyeq$ is a poset-stratified space. 
Recall that the identity $I_\leq =\left(\asp\circ \spp\right)(I_\pi) \to I_\pi$ is continuous. If $\pi \colon X\to I_\preccurlyeq$ is continuous,  then the identity $I_\pi \to I_\preccurlyeq$ is continuous, because $I_\pi$ has the quotient topology induced from the decomposition map $\pi:X \to I$. Composing these two continuous identity maps shows that  the identity $I_\leq \to I_\preccurlyeq$ is continuous. Equivalently, the identity $(I,\leq) \to (I,\preccurlyeq)$ is monotone. Thus the decomposition preorder $\leq$ is a partial order because $i\leq j \leq i \implies i\preccurlyeq j \preccurlyeq i \implies i=j$. Finally, $\pi \colon X \to I_\leq$ is continuous by Proposition \ref{pro-alex} (3) since the decomposition is Alexandrov.

The converse is immediate.

(2) $\implies$ (3): When $\pi \colon X \to I_\leq$ is a poset-stratified space, the stratum 
\[
X_i = \pi^{-1}(i) = \pi^{-1}(U_i \cap D_i) =  \pi^{-1}(U_i) \cap \pi^{-1}(D_i)
\]
where $U_i = \{ j\in I \mid i\leq j\}$ and $D_i = \{j\in I \mid j\leq i\}$ are respectively the minimal open and closed 
neighborhoods of $i$ in the Alexandrov space $I_\leq$.  Therefore $X_i$ is an open subset of $\pi^{-1}(D_i)$. Here we note that it follows from the proof of Lemma \ref{poset-lc} that $\{i\} = U_i \cap D_i$, since the decomposition preorder $\leq$ is a partial order.

(3) $\implies$ (2): Since the decomposition is Alexandrov, the map $\pi \colon X \to I_\pi=I_\leq$ is continuous. Therefore we only need to show that the decomposition preorder $\leq$ is a partial order.  The subspace topology on $D_i$ from $I_\pi$ is the quotient topology induced from the restriction $\pi^{-1}(D_i) \to D_i$ of $\pi$ to its preimage. (This is a well-known general fact\footnote{For example, see Ryszard Engelking's book \cite[2.4.15. Proposition]{Engel}. $D_i$ being a closed set is a key, and in fact, this fact also holds in the case of $\pi^{-1}(B) \to B$ for any open set $B$. However, this fact \emph{does not necessarily hold} if $B$ is neither an open set nor a closed set (e.g., see \cite[2.4.17. Example]{Engel}). A surjective continuous map $f:X \to Y$ is called \emph{hereditarily quotient} if for every subset $B \subset Y$ the restriction $f_B:f^{-1}(B) \to B$ is a quotient map \cite[2.4.F]{Engel}. In this sense we can say that \emph{any quotient map $f:X \to Y$  is hereditarily quotient `with respect to open subsets and closed subsets' of $Y$}. } about the quotient topology.) Therefore, since $X_i=\pi^{-1}(i)$ is open in $\pi^{-1}(D_i)$ by assumption, $\{i\}$ is open in $D_i$. This means that $\{i\}$ is locally closed in $I_\pi=I_\leq$. Hence the  decomposition preorder is a partial order by Lemma \ref{poset-lc} as required.
\end{proof}

The following example shows that Proposition \ref{poset-strat} may fail for a decomposition which is not Alexandrov.
\begin{ex}\label{notalex} Let $X =\{0\} \cup \{\frac{1}{n} \, | \, n \in \mathbb N\}$ be the subspace of the real line $\mathbb R$ with its pointwise decomposition so that the index set $I$ is $X$ itself. The decomposition map $\pi:X \to X$ is the identity $\op{id}_X$ and the quotient space $X_{\pi}=X$. Then:
\begin{enumerate}
\item The pointwise decomposition is {\em not}  Alexandrov as the union $\bigcup_{n\in \mathbb N} \{\frac{1}{n}\} = \{ \frac{1}{n} \, | \, n \in \mathbb N\}$ of closed points is not closed.
\item The points of $X$ are closed so the preorder $\leq$ induced by this topology is discrete, i.e.\ $x \leq y \iff  x=y$. Therefore $X_{\leq} =\left(\asp\circ \spp\right)(X)$ is $X$ with the discrete topology and the identity $\pi=\op{id}_X: X \to X_{\leq}$ is not continuous, i.e.\  Proposition \ref{poset-strat} (2) is false.
\item Nevertheless, Proposition \ref{poset-strat} (1) holds for the partial order on $X$ defined by
$$ x \preccurlyeq y \Longleftrightarrow x =0 \ \text{or}\ x=y.$$
This is because the open sets of $X_{\preccurlyeq}$ are the points $\{\frac{1}{n}\}$ for $n \in \mathbb N$ together with $X$ and $\emptyset$; since each of these is open in $X$ the identity $\op{id}_X: X \to X_{\preccurlyeq}$
is a continuous map.
\end{enumerate}
\end{ex}

\begin{rem}
\label{rem:partial orders on I}
We do not claim, and indeed it is not true, that there is always a unique partial order with respect to which a decomposition is a poset stratification. For example, the two point discrete space $X=\{0,1\}$ with decomposition $X=\{0\}\sqcup\{1\}$ is a poset stratification with respect to any of the three partial orders (distinguished by whether $0\leq 1$ or $0\geq 1$ or neither) on its set $\{0,1\}$ of strata. 

However, the final statement of the proposition implies that the decomposition preorder is initial amongst those partial orders on the set of strata for which the given decomposition is a poset stratification.
\end{rem}

\begin{rem}
\label{lc example without fc}
The third equivalent condition in the above result is stronger than the assumption that each stratum locally closed. For example consider the decomposition
\[
S^1 = \left\{e^{\theta \sqrt{-1}} \mid 0 < \theta \leq \pi \right\} \sqcup \left\{e^{\theta \sqrt{-1}} \mid \pi < \theta \leq 2\pi \right\}
\]
of the unit circle. This is an Alexandrov decomposition, because it is finite, and both strata are locally closed. Nevertheless, the decomposition space is the indiscrete space with two points, which is not the Alexandrov space of any {\em partial order} on the set of strata. So the decomposition is not a poset stratification.
\end{rem}
\section{The Frontier Condition}
\label{sec:frontier condition}
A decomposition $\{X_i \mid i\in I\}$ of $X$ satisfies the \emph{frontier condition}  if 
\[
X_i \cap \overline{X_j} \neq \emptyset \implies X_i \subset \overline{X_j}.
\]
 Equivalently, it satisfies the frontier condition if the closure of each stratum is a union of strata. 

Recall from the previous section that we set $D_j = \{ i\in I \mid i\leq j\} = \overline{\{j\}}$ where the closure is taken in the decomposition space $I_\pi$ and that $\pi^{-1}(D_j)$ is then the minimal closed union of strata containing the stratum $X_j$.
\begin{pro}
\label{FC}
Let $\{X_i \mid i\in I\}$ be an Alexandrov decomposition of $X$. Then the following are equivalent:
\begin{enumerate}
\item the decomposition satisfies the frontier condition,
\item for all $j \in I$, the equality $\overline{X_j} = \pi^{-1}( D_j)$, holds, i.e., the closure of $X_j$ is the minimal closed union of strata containing $X_j$,
\item the decomposition preorder is given by $i \leq j \iff X_i \subset \overline{X_j}$,
\item  the decompostion map $\pi \colon X \to I_{\pi}$ is an open map.
\end{enumerate}
\end{pro}
\begin{proof}
(1) $\implies$ (2):  When the frontier condition holds the closure of a stratum $X_j$  is a union of strata, hence is the minimal closed union of strata containing $X_j$. Therefore $\overline{X_j}=\pi^{-1}(D_j)$, since $D_j$ is the minimal closed subset containing $j$.

(2) $\implies$ (3):  The decomposition preorder is defined by
\[
i\leq j \iff i \in D_j \iff X_i \subseteq \pi^{-1}(D_j) \iff X_i \subseteq \overline{X_j}
\]
where we use (2) at the last step.

(3) $\implies$ (4): Suppose that $U \subseteq X$ is open, and that $i\in \pi(U)$ and $i\leq j$. Then $U \cap X_i \neq \emptyset$ and $X_i \subseteq \overline{X_j}$. So $U\cap X_j \neq \emptyset$, which implies 
 $j\in \pi(U)$ too.  Thus $\pi(U)$ is open in the Alexandrov topology, and because the decomposition is assumed to be Alexandrov this is equivalent to $\pi(U)$ being open in the quotient topology, i.e., open in $I_{\pi}$. Hence $\pi \colon X \to I_\pi$ is an open map as claimed.

(4) $\implies$ (1): Suppose that $\pi \colon X \to I_{\pi}$  is an open map. Hence $\pi \colon X \to I_{\pi}$ is both continuous and open. Recall that\footnote{The following are well-known: (1) $f: X \to Y$ is continuous 
if and only if $\overline {f^{-1}(B)}\subset f^{-1}(\overline B)$ for any $B \subset Y$ and (2) $f: X \to Y$ is open
if and only if $f^{-1}(\overline B) \subset \overline {f^{-1}(B)}$ for any $B \subset Y$.} this implies $\pi^{-1}\left(\overline{J}\right)= \overline{\pi^{-1}(J)}$ for any $J\subseteq I$. In particular $\overline{X_j} = \overline{ \pi^{-1}(j) } = \pi^{-1}( \overline{\{j\}})$ and therefore
\[
X_i \cap\overline{X_j} = \pi^{-1}(i) \cap \overline{ \pi^{-1}(j) } = \pi^{-1}(i)\cap \pi^{-1}\left( \overline{\{j\}}\right) = \pi^{-1}\left( \{i\} \cap \overline{\{j\}}\right).
\]
Thus if $X_i \cap \overline{X_j}\neq \emptyset$, the above implies 
that $\{i\} \cap \overline{\{j\}} \not = \emptyset$, i.e., $i \in \overline{\{j\}}$, which implies that $\pi^{-1}(i) \subset \pi^{-1}( \overline{\{j\}}) =  \overline{\pi^{-1}(j)}$, i.e., 
$X_i \subseteq \overline{X_j}$. Hence the frontier condition holds.
\end{proof}
\begin{rem}\label{ta-ta}
See 
\cite[Lemma 2.3]{Tam}\footnote{Note that the proof of \cite[Proposition 3.2]{TT} is given only for the ``only if" part.} for a similar result about poset-stratified spaces: Dai Tamaki proves that a poset-stratified space $\pi: X \to I_\preccurlyeq$ is an open map if and only if $X_i \subset \overline{X_j} 
 \Longleftrightarrow i 
 \preccurlyeq j$. Here we note that for the proof of ``only if" part he also uses the above key fact that $\pi$ is a \emph{continuous and open} map if and only if $\pi^{-1}\left(\overline{J}\right)= \overline{\pi^{-1}(J)}$ for any $J\subseteq I$.
\end{rem}
\begin{cor}
\label{lc+front}
Let $\{X_i \mid i\in I\}$ be an Alexandrov decomposition of $X$.  The following are equivalent
\begin{enumerate}
\item each stratum $X_i$ is locally closed and the frontier condition holds;
\item $\pi\colon X\to I$ is poset-stratified (for some partial order) and $\pi \colon X \to I_\pi$ is an open map.
\end{enumerate}
Moreover, when these equivalent conditions hold, the decomposition preorder $\leq$ is a partial order and $\pi \colon X \to I_\leq$ is a poset-stratified space.
\end{cor}
\begin{proof}
(1) $\implies$ (2): Suppose that each $X_i$ is locally closed and the frontier condition holds. Then Proposition \ref{FC} implies that $\pi$ is an open map, and that $\overline{X_i} = \pi^{-1}(D_i)$. Since $X_i$ is locally closed, $X_i=U\cap D$ for an open set $U$ and a closed set $D$ in $X$. So, $X_i \subset \overline{X_i} \subset D$, since $X_i \subset D$ and $D$ is a closed set.
Hence we have
$$ X_i \subset U \cap \overline{X_i}  \subset U \cap D =X_i,$$
which implies that $X_i = U \cap \overline{X_i}$, thus $X_i$ is open in $\overline{X_i}$ 
, hence it is open in $\pi^{-1}(D_i)$. Therefore $\pi \colon X \to I$ is poset-stratified for some partial order (in fact with respect to the decomposition preorder) by Proposition \ref{poset-strat}.

(2) $\implies$ (1): Suppose $\pi\colon X\to I$ is poset-stratified and the decomposition map is open. By Proposition \ref{FC} the frontier condition holds. Moreover  $X_i$ is open in $\pi^{-1}(D_i)$ by Proposition \ref{poset-strat}, and so in particular locally closed.
\end{proof}
\begin{rem} The frontier condition plays a key role here. Recall the example in  Remark \ref{lc example without fc}  of an Alexandrov decomposition with locally closed strata (not satisfying the frontier condition),  
which is not poset-stratified. 
\end{rem}
\section{Consequences}
\label{sec:main}

We assemble the results of the previous sections to obtain three theorems. The first states that a stratified space is a poset-stratified space, in a natural way.
{\thmA*} 
\begin{proof}
Since $\{X_i \mid i\in I\}$ is a stratification, it is a locally finite decomposition, with locally closed strata satisfying the frontier condition. As it is locally finite, Corollary \ref{cor:lf is alex} implies that it is an Alexandrov decomposition. As it satisfies the frontier 
condition, Proposition \ref{FC} then shows that the decomposition preorder on the set $I$ of strata is given by 
\[
i \leq j \iff X_i \subseteq \overline{X_j}
\]
and is, in particular, therefore a partial order. Finally, as the strata are locally closed and satisfy the frontier condition, Corollary \ref{lc+front} implies that $\pi \colon X \to I_\leq$ is poset-stratified.
\end{proof}
\begin{rem}
This is `well-known folklore' and variants of it can be found in the literature. For example, see \cite[Corollary 4.8]{Yo} for a version under the extra assumption that there are finitely many strata.
\end{rem}
In the proof of Theorem \ref{thm A},  local finiteness  was used to  obtain that the decomposition is Alexandrov. Note, however, that ``locally finite" can be replaced by ``finitely many strata intersecting each compactum", due to Corollary \ref{cor:compact finite is alex}. 
Namely we can modify Theorem \ref{thm A} as follows:
\begin{cor} Suppose that a decomposition $\{X_i | i \in I \}$ satisfies
\begin{enumerate} 
\item only finitely many strata intersect each compactum, 
\item each stratum is locally  closed and 
\item the frontier condition. 
\end{enumerate}
Then
\begin{enumerate}
\item the set $I$ of strata is partially ordered by $i \leq j \Longleftrightarrow X_i \subset \overline{X_j}$ and
\item the decomposition map $\pi:X \to I_{\pi}=I_{\leq}$ is a poset-stratified space.
\end{enumerate}
\end{cor}
\begin{rem}
The decomposition given in Example \ref{notalex} is \emph{not locally finite}, and even fails the weaker condition that only finitely many strata intersect each compactum. As shown in Example \ref{notalex}, the decomposition map $\pi=\op{id}_X: X \to X_{\leq}$ is not a poset-stratified space with respect to the partial order $\leq$. 
\end{rem}
In contrast, not every decomposition arising from a poset-stratified space defines a stratification. 
The next result gives sufficient conditions.
{\thmB*} 
\begin{proof} Observe first that we may without loss of generality assume that $\pi$ is surjective, and hence that the indexing set of the associated decomposition of $X$ agrees with $I$. To see this, observe that the Alexandrov topology functor $\asp \colon \proset \to \tsp$ maps sub-prosets (equipped with the restricted pre-order) into inclusions of subspaces. Hence, it follows that the induced surjection $X \to \op{Im} \pi$ is continuous with respect to the Alexandrov topology on $\op{Im} \pi$, induced by restricting the preorder $\preccurlyeq$ to $\op{Im} \pi$. Furthermore, as $\op{Im} \pi \to I_{\preccurlyeq}$ is continuous with respect to this topology, it follows that $X \to \op{Im} \pi$ is again open. Using this, let us now assume that the indexing set of the  decomposition associated to $\pi$ agrees with $I$.
Since the decomposition space $I_\pi$ has the quotient topology and $\pi \colon X \to I_\preccurlyeq$ is continuous, there is a commutative diagram
\[
\xymatrix
{
X \ar[d]_{\pi} \ar[dr]^{\pi} & \\
I_\pi \ar[r]_{\id} & I_\preccurlyeq
}
\]
of continuous maps. In particular,  any open set in $I_\preccurlyeq$ is also open in $I_\pi$. Therefore $I_\pi$ is also a locally finite space and $\pi \colon X \to I_\pi$ is also an open map. It follows that the decomposition $\{\pi^{-1}(i) \mid i\in I\}$ is locally finite, and hence Alexandrov by Corollary \ref{cor:lf is alex}. Then, by Corollary \ref{lc+front}, each stratum $\pi^{-1}(i)$ is locally closed  
and the strata satisfy the frontier condition. Thus the decomposition is a stratification, as claimed.
\end{proof}
\begin{rem}
\label{rem:thm B}
When a surjective poset-stratified space $\pi \colon X \to I_\preccurlyeq$ is stratified with respect to {\em some} partial order $\preccurlyeq$ on the set $I$ of strata, Theorem \ref{thm B}  implies the decomposition $\{\pi^{-1}(i) \mid i\in I  
\}$ is a stratification, and then Theorem \ref{thm A} implies that it is poset-stratified {\em with respect to its decomposition preorder}. Moreover, the decomposition preorder is initial amongst those partial orders on $I$ for which it is poset-stratified by Remark \ref{rem:partial orders on I}.
\end{rem}
\begin{rem}
\label{rem:thm B 2} The above proof of Theorem \ref{thm B}   uses the fact that the decomposition is Alexandrov and Corollary \ref{lc+front}, namely, Theorem \ref{thm B}  follows \emph{as a consequence of our previous analysis of Alexandrov decompositions}. There is a direct proof as follows: 
\begin{enumerate}
\item the decomposition is {\em locally finite} because the Alexandrov space $I_\preccurlyeq$ is locally finite.
\item the {\em strata are locally closed} by Lemma \ref{poset-lc}, i.e.\ because each singleton of a poset is locally closed in its Alexandrov space.
\item the {\em frontier condition} follows from $\pi \colon X \to I_\preccurlyeq$ being a continuous and open map, as shown in $(4) \Longrightarrow (1)$ of the proof of Proposition \ref{FC}. See also  \cite[Proposition 3.4]{TT} for a minor variant of this proof.
\end{enumerate}
\end{rem}
The conditions of Theorem \ref{thm B}  are {\em sufficient} for the poset-stratified space to be stratified, but the following example shows that they are not {\em necessary}. 
\begin{ex}
Suppose that $\{ X_i \mid i\in I\}$ is a stratification of $X$. Then $\pi \colon X\to I_\leq$ is poset-stratified by Theorem \ref{thm A} where $i\leq j \iff X_i \subseteq \overline{X_j}$ is the decomposition preorder. 

Let $\preccurlyeq$ be any (strict) refinement of $\leq$, i.e. any different partial order such that $i\leq j \implies i\preccurlyeq j$. Strict refinements exist unless $\leq$ is a total order. Then the identity $\id \colon I_\leq \to I_\preccurlyeq$ is continuous, and therefore $\pi \colon X \to I_\preccurlyeq$ is also poset-stratified. This construction illustrates how poset-stratified spaces over `exotic' partial orders, different from the decomposition preorder, arise. Note that, by construction the decomposition preorder is initial amongst its refinements, hence $I_\leq$ is initial amongst their Alexandrov spaces as claimed in Remark \ref{rem:thm B}.

The Alexandrov space $I_\preccurlyeq$ of a strict refinement of $\leq$ {\em need not be locally finite} and the map $\pi\colon X\to I_\preccurlyeq$ is {\em never open}. (To see the latter, note that there is some $U\subset I$ which is open in $I_\leq$ but not open in $I_\preccurlyeq$. Then $\pi^{-1}(U)$ is open in $X$, but its image $\pi(\pi^{-1}(U)) = U$ is not open in $I_\preccurlyeq$.) Of course, the associated decomposition is unchanged, and so is still a stratification. This shows that the conditions of Theorem \ref{thm B}  are not necessary.

To give a concrete example, let $X=\bN$ with the pointwise decomposition $\{ \{n\} \mid n\in \bN\}$. The decomposition preorder $\leq$ on $\bN$ is trivial, i.e.\ $m\leq n \iff m=n$. Thus any other partial order $\preccurlyeq$ is a refinement, and so yields a poset-stratification $\id \colon \bN \to \bN_\preccurlyeq$ which is not an open map. If, for instance, we choose $\preccurlyeq$ to be the standard order, then $\bN_\preccurlyeq$ is not a locally finite space either.

A more traditional example is to decompose the positive quadrant $X= \bR_{\geq 0}\times \bR_{\geq 0}$ into four strata $X_0=\{(0,0)\}$, $X_1=\bR_{\geq 0}\times \{0\}$, $X_2=\{0\}\times \bR_{\geq 0}$ and $X_3 = \bR_{>0}\times \bR_{>0}$. The decomposition preorder on $I=\{0,1,2,3\}$ is the partial order $0 \leq 1,2 \leq  3$. If we refine this to the standard order $0 \preccurlyeq1 \preccurlyeq 2\preccurlyeq 3$, then $\pi \colon X \to I_\preccurlyeq$ is poset-stratified, but not an open map, because $X_1\cup X_3$ is open, but $\pi(X_1\cup X_3) = \{1,3\}$ is not. 
\end{ex}

Our final result is that the {\em sufficient} conditions of Theorem \ref{thm B}  under which a poset-stratified space is stratified are also {\em necessary} when the partial order is the associated decomposition preorder, i.e.\ the specialization preorder of the quotient topology on the set of strata.
{\thmC*}
\begin{proof}
By assumption the decomposition $\{ \pi^{-1}(i) \mid i\in I\}$ is a poset stratification, hence also an Alexandrov decomposition. Suppose it is a stratification. Then each stratum is locally closed and the frontier condition is satisfied,  so, by Corollary \ref{lc+front}, the map $\pi \colon X \to I_\leq$ is open. Moreover, the decomposition is locally finite,  which implies  that the Alexandrov space $I_\leq$ is also locally finite.

The other direction is a special case of Theorem \ref{thm B}.
\end{proof}
We end with some further examples of poset-stratified spaces which are not stratified, either because the decomposition space is not locally finite or the decomposition map is not open.
\begin{ex}
The Alexandrov space $P_\leq$ of a poset is trivially poset-stratified over itself via the identity $\id_P \colon P_\leq \to P_\leq$. The associated decomposition is the pointwise one $\{ \{p\} \mid p\in P\}$ and the decomposition preorder is the given partial order $\leq$ on $P$. The identity is evidently an open map, so by Theorem \ref{thm C}  the pointwise decomposition of a poset is a stratification if, and only if, $P_\leq$ is locally finite. This is not always the case, for example it fails for the standard order on $\bN$.  

(There is an unfortunate clash of terminology here. A poset $P$ is said to be locally finite if each bounded  interval $[p,q] = \{ r\in P \mid p\leq r\leq q\}$ is finite. This is {\em not} the same as its Alexandrov space $P_\leq$ being a locally finite space, which is equivalent instead to each bounded below interval $[p,\infty)=\{ q\in P \mid p\leq q\}$ being finite. For example, $\bN$ with the standard order is a locally finite poset, but its Alexandrov space $\bN_\leq$ is not a locally finite space.)
\end{ex}
\begin{ex}
Consider the locally finite decomposition $\bR=X_0 \cup X_1$ into locally closed strata $X_0=(-\infty,0]$ and $X_1=(0,\infty)$. The decomposition preorder is the partial order $0 \leq 1$, its Alexandrov space is locally finite (even finite), and $\pi \colon \bR \to \{0,1\}$ is a poset-stratified space. However, the decomposition is not a stratification because $\pi$ is not an open map, equivalently the frontier condition fails.
\end{ex}

\appendix

\section{Semicontinuous decompositions}
\label{sec:semicontinuity}

Theorem B shows that poset-stratified spaces for which  $\pi:X \to I_\preccurlyeq$ is an open map, which implies that the decomposition map $\pi:X \to I_{\pi}$ is also open, have good properties.  It is natural to ask what follows from the decomposition map being \emph{closed}. In this appendix, we point out that these properties have been studied by geometric topologists since the 1920s, in particular by Robert Lee Moore who introduced  notions of semicontinuity and continuity for decompositions (\cite{Moore}, cf. \cite{Moore2}).

 In this appendix we consider a decomposition $\mathcal D$ of a topological space $X$ to be an equivalence relation, and therefore denote the decomposition space by $X/\mathcal D$. The strata $X_i$ are the equivalence classes, that is the points of $X/\mathcal D$. In other words, if we let the decomposition $\mathcal D =\{X_i \, | \, i \in I \}$, then $X/\mathcal D =I$.
\begin{defn}
A decomposition $\mathcal D$ of a space $X$ is
\begin{enumerate}
\item \emph{upper semicontinuous} if $\displaystyle \bigcup_{X_i \cap U \neq \emptyset}  X_i $ is open for any open set $U$ of $X$;
\item \emph{lower semicontinuous} if $\displaystyle \bigcup_{X_i \cap F \neq \emptyset}  X_i$ is closed for any closed set $F$ of $X$;
\item \emph{continuous} if it is both upper and lower semicontinuous.
\end{enumerate} 
\end{defn}
These notions can be phrased in terms of the quotient, or  decomposition, map $X \to X/\mathcal D$ as follows. The decomposition $\mathcal D$ is 
\begin{enumerate}
\item upper semicontinuous if, and only if,  $X \to X/\mathcal D$ is a closed map;
\item lower  semicontinuous if, and only if,  $X \to X/\mathcal D$ is an open map;
\item continuous if, and only if, $X \to X/\mathcal D$ is an open and closed map.
\end{enumerate}

The theory of decomposing a (metric) space into continua (compact connected spaces) was developed by R. L. Moore in the 1920s and later by R.H. Bing in the 1950s (e.g., see \cite{Dav}).  Moore's famous theorem \cite{Moore} is that if $\mathcal D$ is an upper semicontinuous decomposition of the 2-dimensional Euclidean space $\mathbb R^2$ into continua, none of which separates $\mathbb R^2$, then the decomposition space $\mathbb R^2/\mathcal D$ is homeomorphic to the Euclidean space  $\mathbb R^2$.  In contrast, R. H. Bing \cite{Bing} proved that the analogue in $3$ dimensions is false: there exists an upper semicontinuous decomposition $\mathcal D$ of the 3-dimensional Euclidean space $\mathbb R^3$ into continua, none of which separates $\mathbb R^3$, such that the decomposition space $\mathbb R^3/\mathcal D$ is neither homeomorphic to the Euclidean space  $\mathbb R^3$ nor even a manifold, but nevertheless $(\mathbb R^3/\mathcal D) \times \mathbb R^1 \cong \mathbb R^4$, see \cite{Bing2}. This decomposition space $\mathbb R^3/\mathcal D$ is Bing's famous \emph{dogbone space}. 

A similar `wild topology' is used in Michael Freedman's proof of the $4$-dimensional Poincar\'e Conjecture, see for example \cite{Fr1, Fr2}, where again upper semicontinuous decompositions play a key role, namely, \emph{Casson's `kinky' handles}. In contrast, the notion of \emph{lower} semicontinous decomposition does not seem to have been so much studied\footnote{E.g., in \cite[p.10, before Proposition 3]{Dav} Robert Jay Davermann remarks ``Lower semicontinuous decompositions rarely come up in geometric topology, except in conjunction with continuous ones, which do play a role. Neither term will reappear here." Also, in \cite[p.125, \S 2 the first paragraph]{BingPAMS} R. H. Bing remarks ``While upper semicontinuous decompositions have been widely studied, the notion of lower semicontinuous decompositions has not been used nearly so widely."}. However, as shown in the present paper, the notion of \emph{lower} semicontinous decomposition (together with local-finiteness) seems to play a key role in stratification theory, singularity theory, and in algebraic and differential geometry.


\end{document}